\documentclass[a4paper,twoside,openbib,11pt]{article}

\usepackage{amsfonts}
\usepackage{amsmath}
\usepackage{amsthm}
\usepackage{amssymb}
\usepackage[all]{xy}
\usepackage[latin1]{inputenc}
\usepackage{graphicx}
\usepackage{color}
\usepackage{yfonts}

\input xy
\xyoption{all}

\newtheorem{theorem}{Theorem}[section]
\newtheorem{corollary}[theorem]{Corollary}
\newtheorem{lemma}[theorem]{Lemma}
\newtheorem{proposition}[theorem]{Proposition}
\theoremstyle{definition}
\newtheorem{definition}[theorem]{Definition}
\newtheorem{remark}[theorem]{Remark}

\newtheorem{example}[theorem]{Example}

\newcommand{\Ker}{\operatorname{ker}}

\newcommand\iso{\kern.35em{\raise3pt\hbox{$\sim$}\kern-1.1em\to}\kern.3em}

\newcommand{\eq}[1][r]
   {\ar@<-3pt>@{-}[#1]
    \ar@<-1pt>@{}[#1]|<{}="gauche"
    \ar@<+0pt>@{}[#1]|-{}="milieu"
    \ar@<+1pt>@{}[#1]|>{}="droite"
    \ar@/^2pt/@{-}"gauche";"milieu"
    \ar@/_2pt/@{-}"milieu";"droite"}

\title{Linear representations of convolutional codes over rings
\thanks{Partially supported by INCIBE. Ministerio de Industria, Spain.
 }
\author{Miguel V. Carriegos\footnote{RIASC, Universidad de Le\'on, SPAIN, mail to: {\texttt miguel.carriegos@unileon.es}}\qquad Noem\'{i} DeCastro\footnote{Departamento de Matem\'aticas, Universidad de Le\'on, SPAIN, mail to: {\texttt ncasg@unileon.es}} \qquad Ángel Luis Muñoz Castañeda\footnote{Institut f\"ur Mathematik, Freie Universit\"at, Berlin, GERMANY, mail to: {\texttt angel@math.fu-berlin.de}}}
}
\begin{document}

\maketitle

\begin{abstract}
In this paper we extend the relation between convolutional codes and linear systems over finite fields to certain commutative rings through first order representations . We introduce the definition of \emph{rings with representations} as those for which these representations always exist, and we show that finite products of finite fields belong to this class. We develop the input/state/output representations for convolutional codes over these rings, and we show how to use them to construct observable convolutional codes as in the classical case.
\end{abstract}



\section{Introduction}
Convolutional codes are  error-correcting codes used to detect and correct sets of digital data. Convolutional codes over finite fields were introduced by Peter Elias in 1955 %although the first algebraic-theoric approach of convolutional codes was given by Forney in \cite{forney} 
and, in the current context, a considerable research in this field is developed by using algebraic, combinatorics, computer science, control theoretic or algebro-geometric tools among others (see \cite{fragouli}, \cite{GLROS}, \cite{HUT}, \cite{MC}, \cite{ravi2} or \cite{zerz}).

The first approach to convolutional codes over rings was given by Massey and Mittelholzer in \cite{Massey, Masseymit}. There is a considerable body of literature about convolutional codes over rings where generator matrices, minimal encoders and their properties have been studied (see \cite{Fagnani, JH}). Moreover, trellis representations and properties of convolutional codes over $\mathbb{Z}/p^{r}\mathbb{Z}$ are developed in \cite{PINTO,PINTOKUIJPOLDER}.

We are interested in the approach to convolutional codes over finite fields by linear systems. This relation is given in terms of first order representations of the code, that is, triples of matrices $(K,L,M)$  that allow us to obtain another set of matrices $(A,B,C,D)$ that forms a reachable (controllable) input/state/output (I/S/O) representation  of the convolutional code, where the inputs and outputs of a system are part of the codeword (the main results can be found in \cite{KUIJ}, \cite{BCH}, \cite{RSY}, and \cite{York}).  Moreover, in  \cite{isabelhijo} it is shown that the decoder process of the code is given by the output controllability matrix (the matrix which solves the associated linear dynamical system). 

The natural question is whether we can generalize the above duality to certain commutative rings with identity. Within this goal, we introduce the definition of \emph{rings with representations} generalizing the above described relation between codes and systems to this class of rings and their finite products.

This paper is organized as follows. In section 2 we give some algebraic preliminaries that are needed in the rest of the paper. In section 3 we define the concept of family of convolutional codes over a ring $R$, and we develop for them the basic theory of first order representations generalizing the classical case. We then define the class of rings with representations and we show that finite products of finite fields belong to this class. In section 4 we show the existence of I/S/O representations for families of convolutional codes over finite products of finite fields and their reachability properties. In section 5 we use the above results to construct observable families of convolutional codes from linear system point of view. Finally we give our conclusions and further research.

%The paper is organized as follows: in Section I we give some preliminaries of algebraic results and some known notions about convolutional codes and its relation with linear systems over finite fields. In Section II, we generalize the definition of convolutional code to certain commutative rings with identity obtaining \emph{families of convolutional codes}. In Section III we introduce the definition of \emph{rings with realization} and we show that there exist minimal first order representations of the families of convolutional codes over this class of rings. Furthermore, in Section IV we compute the associated  I/S/O representations to above families. In Section V, we study observable families of convolutional codes and we apply our results to construct observable families. Finally, our conclusions.

\section{Preliminaries}

We first give a brief overview of the theory of convolutional codes over a finite field. Finally we state the basic algebraic preliminaries that will be used in the rest of the paper.

\subsection{Convolutional Codes over Finite Fields} 

Let us start by recalling some basic definitions and known results regarding convolutional codes and their representations, the reader can see \cite{BCH}, \cite{RSY} and  \cite{York} as main references on the topic.

Let $\mathbb{F}$ be a finite field and $k\leq n\in\mathbb{N}$. A $(n,k)$ convolutional code over $\mathbb{F}$ is a rank $k$ submodule $\mathcal{C}\subset \mathbb{F}[z]^{n}$. Any matrix $G(z)\in Mat_{n\times l}(\mathbb{F}[z])$, with $l\geq n$, whose columns generate $\mathcal{C}$ is called a generator matrix of $\mathcal{C}$. A generator matrix $G(z)$ of $\mathcal{C}$ of size $n\times k$ is called an encoder of $\mathcal{C}$. Note that any encoder is necessarily injective.

Let $G(z)$ be an encoder for a $(n,k)$ convolutional code $\mathcal{C}$ and denote $\underline{g}_{j}(z)=(g_{i,j}(z))_{i=1,\hdots n}$ the jth column of $G(z)$. The column degree of the encoder is defined as the maximum degree of its coordinates, $\nu_{j}:={max}\{deg(g_{i,j}(z))| i=1,\hdots n\}$. Reordering the columns if were necessary, we may assume that $\nu_{1}\geq \hdots \geq \nu_{k}$. The complexity, or degree, of the encoder is then defined as $\delta(G(z)):=\sum_{j=1}^{k}\nu_{j}$ while the memory is defined as the maximum column degree, i.e. $\nu_{1}$. Note that a memoryless convolutional code is a block code. The complexity, or the degree, of the code $\mathcal{C}$, $\delta(\mathcal{C})$, is the highest degree of the full size minors of any encoder $G(z)$ of $\mathcal{C}$. An encoder $G(z)$ is called minimal if $\delta(G(z))=\delta(\mathcal{C})$.

Let $\mathcal{C}$ be a $(n,k)$ convolutional code over $\mathbb{F}$ with degree $\delta$. A first order representation of $\mathcal{C}$ is a triple of matrices $(K,L,M)$ with $K,L\in Mat_{\delta+n-k\times\delta}(\mathbb{F})$ and $M\in Mat_{n-k+\delta\times n}(\mathbb{F})$ such that 
$$\mathcal{C}=\{v(z)\in\mathbb{F}[z]^{n} \ | \ \exists x(z)\in\mathbb{F}[z]^{\delta} \textit{ such that }zKx(z)+Lx(z)+Mv(z)=0\}$$
Moreover, if the representation satisfies the following conditions,
\begin{enumerate}
\item $K$ has column full size rank,
\item $(K,M)$ has row full size rank,
\item $rk(z_{0}K+L,M)=\delta+n-k$ for all $z_{0}\in\overline{\mathbb{F}}$, being $\overline{\mathbb{F}}$ the algebraic clousure
\end{enumerate}
then it is called minimal. And, on the other hand, two first order representations $(K,L,M)$ and $(K',L',M')$ are equivalent if there exist (unique) invertible matrices $T,S$ of the adequate sizes such that $(K',L',M')=(TKS^{-1},TLS^{-1},TM)$. The main theorem regarding first order representations is
\begin{theorem}\emph{(\cite[Th. 5.1.1, Th. 5.1.4]{York}, \cite[Th. 3.1, Th. 3.4]{RSY})}\label{york}
Every convolutional code $\mathcal{C}$ admits a unique (up to equivalence) minimal first order representation $(K,L,M)$.
\end{theorem}

Let $\mathcal{C}\in\mathbb{F}[z]^{n}$ be a $(n,k)$ convolutional code  of degree $\delta$ and $(K,L,M)$ a minimal first order representation. We know that the matrix $(K,L,M)$ has full rank so there is an invertible matrix $W$ of rank $\delta+n-k$ such that, reordering the code words if were necessary, it holds
$$W(K,L,M)=(\mathcal{K},\mathcal{L},\mathcal{M})$$
where
\begin{equation}\label{iso}
\mathcal{K}=\left(\begin{array}{cc}
-Id_{\delta}\\
O \\
\end{array}
\right), \mathcal{L}=\left(\begin{array}{cc}
A \\
C \\
\end{array}
\right) \textit{ and }  \mathcal{M}=\left(\begin{array}{cc}
O & B  \\
-Id_{(n-k)} & D\\
\end{array}
\right)
\end{equation}
The matrices $A, B, C$ and $D$ over $R$ obtained from (\ref{iso}) form an I/S/O representation of $\mathcal{C}$, that is, they define a linear system with state-space realization given by
%\begin{equation} \label{rosa}
%\mathfrak{C}=\{v(z) \in R[z]^{n} \mid \exists x(z) \in R[z]^{\delta} \textit{ such that } (zK+L)x(z)+Mv(z)=0 \}
%\end{equation}
%to
\begin{equation}
\label{dinamic2}
\left\{\begin{array}{l} \overrightarrow{x}_{t+1}=A\overrightarrow{x}_t+B\overrightarrow{u}_t\\ \overrightarrow{y}_t=C\overrightarrow{x}_t+D\overrightarrow{u}_t\\
\overrightarrow{v}_t=\left(\begin{array}{c}\overrightarrow{y}_t\\\overrightarrow{u}_t\end{array}\right)\: , \:
x_0=0,\: \exists \gamma: \overrightarrow{x}_{\gamma+1}=0.
\end{array}\right.
\end{equation}
 where $\overrightarrow{x}(t)$ is the $n$-state vector, $\overrightarrow{y}(t)$ the $p$-vector output and $\overrightarrow{u}(t)$ the $m$-vector control. We also give an initial state $x_{t_{0}}=x_{0}$ in time $t_{0}$. 
\begin{remark}
Note that this linear system is reachable (by minimality conditions of first order representations). Conversely if such a reachable linear system is also observable then the associated convolutional code is observable (see \cite{RSY, BCH}).
\end{remark}

%\subsection{Algebraic preliminaries}

\subsection{Kernel of a Pair of Morphisms}

 Let us recall some definitions and properties from \cite{nuestroarxiv}. These allow us to study systematically $R$-modules defined in the same manner as in property (4) stated above.

Let $R$ be a commutative ring (with unit), and let $M_{1},M_{2},N$ be $R$-modules. Let $f_{i}:M_{i}\rightarrow N$ be two $R$-linear maps. We define the kernel $Ker(f_{1}| f_{2})$ as
\begin{equation*}
Ker(f_{1}|f_{2}):=\{m_{2}\in M_{2}|\exists m_{1}\in M_{1}:f_{1}(m_{1})+f_{2}(m_{2})=0\}
\end{equation*}
There are three different ways to present the kernel of two $R$-linear maps,
\begin{enumerate}
\item $Ker(f_{1}|f_{2})= f_{2}^{-1}(Im(f_{1}))$
\item $Ker(f_{1}|f_{2})\simeq Coker(Ker(f_{1})\hookrightarrow Ker(f_{1},f_{2}))$
\item $Ker(f_{1}|f_{2})= Ker(p_{1}\circ f_{2}:M_{2}\rightarrow N/Im(f_{1}))$
\end{enumerate}
(being the projection $p_1:N\rightarrow N/Im(f_1)$) from which the main properties are derived. For instance, $Ker(f_{1}|f_{2})$ behaves well with respect to flat base change. As direct consequence, we have the following particular situation,
\begin{corollary}\emph{(\cite[Corollary 5.5]{nuestroarxiv})}\label{restriction}
Consider matrices $A(z), B(z)$ of adequate sizes, $p\times q_1$ and $p\times q_2$ respectively, and entries in $R[z]$ where $R=R_1\times\cdots\times R_t$ is a product ring with structural idempotents $e_i\in R$. Then one has
$$
Ker(A(z)\mid B(z))=\{u(z)\mid\exists x(z):A(z)x(z)+B(z)u(z)=0\}=
$$
$$
=e_1Ker(\pi_1(A)(z)\mid\pi_1(B)(z))+\cdots+e_t Ker(\pi_t(A)(z)\mid\pi_t(B)(z))
$$
being $\pi_i$ the ith projection of $R$ onto $R_{i}$.
\end{corollary}

\subsection{Rank of a Matrix with Coefficients in a Ring }

We assume that any commutative ring is a ring with unit and any morphism of rings $\phi:R\rightarrow S$ maps the identity to the identity.

Let $R$ be a commutative ring and let $A\in Mat_{n\times m}(R)$ be a matrix with coefficients in $R$. For any positive integer $0<i\leq r:=min\{n,m\}$, we define 
$$\mathcal{U}_{i}(A):= \textit{ideal generated by the }i\times i\textit{ minors of }A$$
These ideals form a chain 
$$(0)\subseteq \mathcal{U}_{r}(A)\subseteq \mathcal{U}_{r-1}(A)\subseteq\hdots\subseteq \mathcal{U}_{1}(A)\subseteq R$$
The main property of the ideals $\mathcal{U}_{i}(A)$ is that they are stable under base change, i.e. if $g:R\rightarrow S$ is a morphism of rings  then 
$$\mathcal{U}_{i}(A\otimes 1)=\mathcal{U}_{i}(A)\cdot S$$
As a trivial consequence we have
\begin{lemma}
Let $g:R\rightarrow S$ be a morphism of rings and let $A\in Mat_{n\times m}(R)$ be a matrix. 
\begin{enumerate}
\item[i)] If $\mathcal{U}_{i}(A)=R$ then $\mathcal{U}_{i}(A\otimes 1)=S$. 
\item[ii)] If $n=m$, then $det(A\otimes 1)$ is invertible in $S$ if $det(A)$ is invertible in $R$.
\end{enumerate}
\begin{proof}
(i) If $\mathcal{U}_{i}(A)=R$ then $1_{S}\in g(\mathcal{U}_{i}(A))$, so $\mathcal{U}_{i}(A\otimes 1)=S$ holds. (ii) The second part follows  from the fact that $A\otimes 1=g^{*}(A)$, and that $g$ maps invertible elements to invertible elements.
\end{proof}
\end{lemma}
Consider the corresponding chain of annihilators 
$$(0)=Ann_{R}(R)\subseteq Ann_{R}(\mathcal{U}_{1}(A))\subset \hdots \subseteq Ann_{R}(\mathcal{U}_{r}(A))\subseteq R$$
We define
\begin{definition}
The (determinantal) rank of the matrix $A$ is defined as 
$$rk(A):=max\{i \ | \ Ann_{R}(\mathcal{U}_{i}(A))= 0\}$$
\end{definition}
As in the classical case, if $A\in Mat_{n\times n}(R)$ then the equation $AX=0$ has no non zero solutions if and only if $rk(A)=n$ (McCoy's Theorem). Despite  that, the condition $rk(A)=n$ does not mean $A$ is invertible. In fact, it can be easily shown that $rk(A)=n$ if and only if $det(A)$ is not a zero divisor. Therefore we have,

\begin{lemma}\emph{(\cite[Theorem 2.1]{BBvV})}\label{surin}
Let $A\in Mat_{n\times m}(R)$ be a matrix with $n\geq m$. Then $A\textit{ is injective }$ if and only if $rk(A)=m$. Let $A\in Mat_{n\times m}(R)$ be a matrix with $n\leq m$. Then $A$ is surjective if and only if $\mathcal{U}_{n}(A)=R$.
\end{lemma}
For the sake of clarity, we fix first of all some notation. For any prime ideal $\mathfrak{p}\in R$ we denote $k(\mathfrak{p})=R_{\mathfrak{p}}/\mathfrak{p}R_{\mathfrak{p}}$ the residue field and for any $R[z]$-module $M$ we denote $M(\mathfrak{p})$ the $k(\mathfrak{p})[z]$-module $M/\mathfrak{p}M=M\otimes_{R[z]}k(\mathfrak{p})[z]$.

Let $A\in Mat_{n\times m}(R)$  be a matrix with $n\leq m$ and $M=Coker(A)$ then, by Nakayama's Lemma, we have
$$M_{\mathfrak{p}}=0\Leftrightarrow M(\mathfrak{p})=0\Leftrightarrow rk(A(\mathfrak{p}))=n$$
for any prime ideal $\mathfrak{p}$ of $R$. Therefore,

\begin{lemma}\label{fullrank}
Let $A\in Mat_{n\times m}(R)$ be a matrix with $n\leq m$. Then $A\textit{ is surjective }$ if and only if $rk(A(\mathfrak{p}))=n, \ \forall \mathfrak{p}\in Spec(R)$.
\end{lemma}
%We say that a commutative local ring $(R,\mathfrak{m},k)$ satisfies \emph{Krull condition} if 
%$$\bigcap \mathfrak{m}^{i}=(0)$$
%and we say that a commutative ring $R$ satisfies \emph{Krull condition} if every local ring $R_{\mathfrak{p}}$ does. Thanks to Krull IntersectionTheorem, we already know that any noetherian ring satisfies Krull condition.
Recall that a local ring $R$ has Krull dimension $0$ if and only if every element of its maximal ideal is nilpotent, i.e. $\mathfrak{m}=Nil(R)$. We can show that many results regarding ranks of matrices over fields can be easily translated to the case of matrices over a commutative ring of dimension $0$,
\begin{proposition}\label{ale}
Let $R$ be a commutative ring of dimension $0$. An element $a\in R$ is invertible if and only if $a$ is not a zero divisor.
\end{proposition}
\begin{proof}
If $a$ is not invertible then there is a maximal ideal $\mathfrak{m}\subset R$ with $a\in \mathfrak{m}$ and therefore $a/1$ belongs to the maximal ideal of the local ring $R_{\mathfrak{m}}$, thus there is a natural number $n$ with $a^{n}/1=0$, so there exists an element $b\in R\setminus \mathfrak{m}$ with $ba^{n}=0$. Assume $n$ is the smallest satisfying this condition, so $ba^{n-1}\neq 0$. Then $(ba^{n-1})a=0$, so $a$ is a zero divisor. The converse is trivial.
\end{proof}
\begin{corollary}
Let $R$ be a commutative ring of dimension $0$ and let $A\in Mat_{n\times n}(R)$. Then $A$ is invertible if and only if $rk(A)=n$.
\end{corollary}
\begin{proof}
Since $rk(A)=n$ if and only if $det(A)$ is not a zero divisor, the result follows from Proposition \ref{ale}.
\end{proof}

Another case of interest is $R=\mathbb{F}[z]$, being $\mathbb{F}$ a field. Since the units of the ring of polynomials $\mathbb{F}[z]$ are the nonzero constants, we have
\begin{proposition}\label{surpol}
Let $A(z)\in Mat_{n\times m}(\mathbb{F}[z])$ be a matrix with $n\leq m$. Then $A(z)$ is surjective if and only if $A(z_0)$ has rank $n$ fo all $z_0\in\overline{\mathbb{F}}$.
\end{proposition}
\begin{proof}
By Lemma \ref{surin}, $A(z)$ is surjective if and only if $\mathcal{U}_n(A(z))=\mathbb{F}[z]$. Let $f_1(z),\hdots,f_l(z)$ be the set of maximum size minors of $A(z)$. Then the above condition is equivalent to $gcd(f_1(z),\hdots,f_l(z))=1$, which is equivalent to the condition $\{f_1(z),\hdots,f_l(z)\}$ have no common roots in $\overline{\mathbb{F}}$.
\end{proof}

\section{Families of Convolutional Codes and their First Order Representations}

\subsection{Families of Convolutional Codes}

Let $R$ be a commutative ring. 

\begin{definition}
A  $(n,k)$ convolutional code over $R$ is a free submodule $\mathfrak{C}\subset R[z]^{n}$ of rank $k$, and such that $R[z]^{n}/\mathfrak{C}$ is flat over $R$.
\end{definition}

Note that  flatness of the quotient $R[z]^{n}/\mathfrak{C}$ allow us to interpret $\mathfrak{C}$ as a family of convolutional codes parametrized by $Spec(R)$. Otherwise, $\mathfrak{C}(\mathfrak{p})$ might not be a submodule of $R[z]^{n}$ anymore. In this setting, the complexity, or the degree, of the code $\mathfrak{C}$ is no longer an integer but a function $\delta:Spec(R)\rightarrow \mathbb{N}$. 
\begin{remark}
In the rest of the paper, we will assume that the degree function $\delta$ is a constant.
\end{remark}
Wen define generator matrices and encoders for families of convolutional codes, following the classical case.
\begin{definition}
A generator matrix $G(z)$ of a  $(n,k)$ convolutional code over $R$, $\mathfrak{C}$, is a matrix $G(z)\in Mat_{n\times l}(R[z])$, with $l\geq k$, such that $Im(G(z))=\mathfrak{C}$. A generator matrix of size $n\times k$ is called an encoder.  
\end{definition}
By  \cite[Prop. 1.3]{vas}, it follows that any encoder is injective, as in the usual case.
\begin{example}
Consider the ring $R=\mathbb{Z}/6\mathbb{Z}$ and the matrix 
\begin{equation*}
G(z)=\left(\begin{array}{cc}
z+3 & 5 \\
3z^{2}+1 & -2z+2 \\
-z^{2}+4z-1 & 3z-3
\end{array}
\right)
\end{equation*}
Note that the full size minors are $\{z^{2}+2z+1,2z^{2}+4z+2,z^{3}+z^{2}+5z+5\}$. Since $Ann_{R}(\mathcal{U}_{3}(G(z)))=0$, we conclude that $G(z)$ is injective (see Lemma \ref{surin}), so its image, $\mathfrak{C}\subset R[z]^{3}$, is a free $R[z]$-module. Since $R[z]^{3}/\mathfrak{C}$ is $R$-flat (R is an absolutely flat ring), we deduce that $\mathfrak{C}$ is a $(3,2)$ family of convolutional codes over $R$ and $G(z)$ is an encoder for $\mathfrak{C}$.
\end{example}
\subsection{Observability}
A very important property for classical convolutional codes is observability. Recall that given a  $(n,k)$ convolutional code $\mathcal{C}\subset\mathbb{F}[z]^{n}$, we say that it is observable if there exists a surjection (the so called syndrome Former) $\psi:\mathbb{F}[z]^{n}\twoheadrightarrow\mathbb{F}[z]^{n-k}$ such that $Ker(\psi)=\mathcal{C}$ (see \cite[Lemma 3.3.2]{York}). Note that this is the same as saying that the quotient $\mathbb{F}[z]^{n}/\mathcal{C}$ is a flat $\mathbb{F}[z]$-module. Therefore we define
\begin{definition}
Let $\mathfrak{C} \subset R[z]^{n}$ be a family of convolutional codes over $R$. We say that $\mathfrak{C}$ is observable if the quotient $R[z]^{n}/ \mathfrak{C}$ is flat over $R[z]$.
 \end{definition}
Since flatness is preserved after base change it turns out from the definition that any member $\mathfrak{C}(\mathfrak{p})$ of the family is observable. 
\begin{example}
Consider the ring $R=\mathbb{Z}/6\mathbb{Z}$ and let 
\begin{equation*}
G(z)=\left(\begin{array}{cc}
z+3 & 5 \\
3z^{2}+1 & -2z+2 \\
-z^{2}+4z-1 & 3z-3
\end{array}
\right)
\end{equation*}
be the matrix considered in above example. As we have shown, the image defines a family of convolutional codes and $G(z)$ is an encoder. Let us show that the family of convolutional codes defined by $G(z)$ is not observable. Since flatness is preserved after base change, it is enough to show that the restriction of the quotient $R[z]^{3}/Im(G(z))$ to $\mathbb{Z}/2\mathbb{Z}$ is not flat.  Note that the restriction of $G(z)$ modulo $2$ is given by
\begin{equation*}
G_1(z)=\left(\begin{array}{cc}
z-1 & 1 \\
z^{2}+1 & 0 \\
z^{2}+1 & z+1
\end{array}
\right)
\end{equation*}
An easy computation shows that 
$$\mathcal{U}_{2}(G_{1}(z))=<z^{2}+1,(z^{2}+1)(z+1)>$$ 
Since $gcd(z^{2}+1,(z^{2}+1)(z+1))=z^{2}+1$, we have 
$$\mathcal{U}_{2}(G_{1}(z))=(z^{2}+1)\neq R$$ 
so $(\mathbb{Z}/2\mathbb{Z})[z]^{3}/Im(G_{1}(z))$ is not flat over $(\mathbb{Z}/2\mathbb{Z})[z]$. Since $R[z]^{3}/Im(G(z))$ is $R$-flat, 
$$(\mathbb{Z}/2\mathbb{Z})[z]^{3}/Im(G_{1}(z))\simeq R[z]^{3}/Im(G(z))\otimes_{R}\mathbb{Z}/2\mathbb{Z}$$
so, we conclude that $R[z]^{3}/Im(G(z))$ is not flat over $R[z]$.
%\end{equation*}

%Then,
%\begin{enumerate}
%\item Since $G(z)$ is injective it defines an encoder for the family of convolutional codes defined by the submodule spanned by its columns.
%\item $R[z]^{3}/Im(G(z))$ is $R$-flat because $R=\mathbb{Z}/2\mathbb{Z}\times\mathbb{Z}/3\mathbb{Z}$...
%\item But $R[z]^{3}/Im(G(z))$ is not free as $R[z]$-module because 
%\begin{equation*}
%G(z)\sim \left(\begin{array}{cc}
%z+3 & 1 \\
%z^{2}+2z+1 & 0 \\
%0 & 0
%\end{array}
%\right)\sim \left(\begin{array}{cc}
%0 & 1 \\
%z^{2}+2z+1 & 0 \\
%0 & 0
%\end{array}
%\right)
%\end{equation*}
%and the rank of $G(z)$ depends on $z\in\mathbb{Z}/6\mathbb{Z}$.
%\item It is not $R[z]$-flat neither.
%\end{enumerate}
\end{example}
The particular case we are interested in is when $R$ is a finite product of commutative rings, $R=R_{1}\times\hdots \times R_{t}$. In this case we have,
\begin{proposition}
\label{observableiff}
$\mathfrak{C}$ is observable $\Leftrightarrow$ $\mathcal{C}_{j}$ is observable $\forall$ $j$.
\end{proposition}
\begin{proof}
The direct implication is clear. Let us show the converse.  If $\mathcal{C}_{j}$ is observable, then $R_{j}[z]^{n}/\mathcal{C}$ is a $R_{j}[z]$- flat module for $j=1\ldots,t$. For each $j$ we can consider the following exact sequence
\begin{equation}
 \label{cuatrocuatro}
 0 \longrightarrow \mathcal{C}_{j} \hookrightarrow R_{j}[z]^{n} \longrightarrow R_{j}[z]^{n}/\mathcal{C}_{j} \longrightarrow 0
 \end{equation}
and then we can construct the following exact sequence
$$
 \xymatrix{
 0 \ar[r]& \bigoplus_{j}^{t} \mathcal{C}_{j} \ar@{^(->}[r]& \bigoplus_{j}^{t} R_{j}[z]^{n} \ar[r]& \bigoplus_{j}^{t} (R_{j}[z]^{n}/\mathcal{C}_{j}) \ar[r] & 0 \\
 0 \ar[r]& \mathfrak{C} \ar@{^(->}[r]\eq[u] & R[z]^{n} \ar[r]\eq[u] & R[z]^{n} / \mathfrak{C} \ar[r]\eq[u]&  0
  }
$$
Since flatness is stable under finite products, $\bigoplus_{j} (R_{j}[z]^{n}/\mathcal{C}_{j})$ is a flat $R[z]$-module. So $R[z]^{n} / \mathfrak{C}$ is a flat $R[z]$-module and $\mathfrak{C}$ is observable.
\end{proof}
\subsection{Minimal First Order Representations of Families of Convolutional Codes}
Let us define now first order representations of convolutional codes over a commutative ring $R$,
\begin{definition}
A  first order representation of a $(n,k)$ family of convolutional codes of degree $\delta$ over $R$ is a triple of matrices $(K,L,M)$ with $K,L\in Mat_{\delta+n-k\times \delta}(R)$ and $M\in Mat_{\delta+n-k\times n}(R)$ satisfying $\mathfrak{C}=Ker(f_{1}|f_{2})$ (where $f_{2}=zK+L$ and $f_{1}=M$). Moreover, if the following conditions are verified,
\begin{enumerate}
\item the matrix $K:R^{\delta}\rightarrow R^{\delta+n-k}$ is  injective with flat cokernel
\item the matrix $(K,M):R^{\delta+n}\rightarrow R^{\delta+n-k}$ is surjective
\item the matrix $(zK+L,M): R[z]^{\delta+n}\rightarrow R[z]^{\delta+n-k}$ is surjective
\end{enumerate}
then it is called minimal.
\end{definition}
\begin{remark}
Note that in case $R=\mathbb{F}$, above conditions agree with the classical ones: (1) and (2) are trivial and (3) follows easily by Proposition \ref{surpol}. Note also that the equivalence relation we had for fields can be adapted to this general setting in the obvious way.
\end{remark}
The central definition of this paper is the following,
\begin{definition}
A commutative ring $R$ is a \emph{ring with representations} if every convolutional codes $\mathfrak{C}$ over $R$ has a unique (up to equivalence) minimal first order representation.
\end{definition}
After giving these definitions an important question comes up: what conditions on the base ring $R$ ensure us the existence of minimal first order representations for a fixed convolutional code? It is well known the existence for classical convolutional codes, i.e., for $R=\mathbb{F}$ a finite field. 

In this section we will show that if $R$ is a finite product of rings, $R_1,\hdots,R_t$, then it is enough to prove the existence (and uniqueness) of minimal first order representations for each $R_i$.

Let $R$ be a commutative ring splitting into a finite product of rings $R \simeq R_{1}\times\hdots\times R_{t}$ and denote by 
\begin{equation*}
I_{j}=R_{1}\times\hdots\times\widehat{R}_{j}\times\hdots\times R_{t}
\end{equation*}
the ideal generated by all the components except $R_{j}$. Then we have an exact sequence
\begin{equation*}
0\rightarrow I_{j}\hookrightarrow R\rightarrow R_{j}\rightarrow 0
\end{equation*}
for all $j=1,\hdots, t$.
First we have the following basic properties:

\begin{lemma}
\label{lema}
\label{properties}
We consider the ring $R=R_{1}\times\hdots\times R_{t}$ and let $A$ be a matrix over $R$. Let us denote by $A_{j}:=A \ mod(I_{j})$. Then, the following holds:
\begin{enumerate}
\item $A_{j}=e_{j}A$.
\item There are isomorphisms $Mat_{n\times m}(R)\simeq \prod_{j=1}^{t} Mat_{n\times m}(R_{j})$ and $Gl_{n}(R)\simeq\prod_{j=1}^{t} Gl_{n}(R_{j})$ given by $A\rightarrow (e_{j}A)_{j=1,\hdots, t}$.
\item If $A_j$ are matrices whose rows are free over $R_{j}$ for each $j$, then $A$ is a matrix whose rows are free over $R$.
\item If $A_{j}$ are matrices whose columns are free over $R_{j}$ for each $j$, then $A$ is a matrix whose columns are free over $R$.
%\item If $A_{j}$ are invertible matrices over $R_{j}$ for each $j$, then $A$ is an invertible matrix over $R$.
\end{enumerate}
\end{lemma}

%Assume once and for all that our ring admit a decomposition $R=R_{1}\times\hdots\times R_{t}$ and that there are first order representations for families of convolutional codes over each ring $R_{i}$. Then we have:
\begin{proposition}{(Existence)}
\label{existencia2}
Let $R$ be a commutative ring and assume that there is a decomposition $R \simeq R_{1}\times\hdots\times R_{t}$ such that $R_{j}$ is a ring with representations for all $j$. Then there exists a minimal first order representation for any $(n,k)$ family of convolutional codes, $\mathfrak{C}\subset R^{n}[z]$, of degree $\delta$.
\end{proposition}
\begin{proof}
Let $\mathcal{C}_{j}\simeq \mathfrak{C}\otimes_{R} R_{j}$ be the restriction of $\mathfrak{C}$ to $R_{j}$. From the definition of $\mathfrak{C}$ we know that $\mathcal{C}_{j}$ is a $(n,k)$- convolutional code over $R_{j}$ with constant degree $\delta$. Then, there are matrices $K_{j},L_{j} \in Mat_{(\delta +n-k) \times \delta}(R_j)$ and $M_{j}\in Mat_{(\delta+n-k) \times n}(R_j)$ such that $\mathcal{C}_{j}=\Ker(f_{1}^{j}|f_{2}^{j})$ where $f_{1}^{j}=zK_{j}+L_{j}$ and $f_{2}^{j}=M_{j}$. Let $e_{1},\hdots,e_{t}$ be the structural idempotents and consider $K=\sum_{j=1}^{t}K_{j}e_{j}$, $L=\sum_{j=1}^{t}L_{j}e_{j}$ and $M=\sum_{j=1}^{t}M_{j}e_{j}$. We define $f_{1}=zK+L$ and $f_{2}=M$. Then, we clearly have $\mathcal{C}_{j}=\Ker(f_{1}\otimes e_{j}|f_{2}\otimes e_{j})$ (see Corollary \ref{restriction}). Since $\mathfrak{C}\simeq\oplus_{j=1}^{t}\mathcal{C}_{j}e_{j}$ we finally get
\begin{equation*}
\mathfrak{C}=\Ker(f_{1}|f_{2})
\end{equation*}
The fact that $K,L,M$ satisfy minimality conditions 1) and 2) follows from Lemma \ref{lema}. The third condition of minimality is trivial since the matrix $(f_{1},f_{2})$ is surjective on each component $R_{j}$.
\end{proof}

\begin{proposition}{(Uniqueness)}
\label{unicidad2}
Let $R$ be a commutative ring as in Proposition \ref{existencia2}. Moreover, assume that first order representations of each $\mathcal{C}_{j}$ are unique up to equivalence over all $R_{j}$. Then for any $(n,k)$ family of convolutional codes $\mathfrak{C}$ over $R$ of degree $\delta$ and any two first order representations $(K,L,M)$ and $(K',L',M')$ of $\mathfrak{C}$ there exist invertible matrices over $R$, $S$ and $T$, such that
\begin{equation*}
(K',L',M')=(TKS^{1},TLS^{1},TM)
\end{equation*}
\end{proposition}
\begin{proof}
Let $(K,L,M)$ and $(K',L',M')$ be first order representations of the family of convolutional codes  $\mathfrak{C}\subset R^{n}[z]$ and $(K_{j},L_{j},M_{j})$, $(K'_{j},L'_{j},M'_{j})$ their restrictions to $R_{j}$. By Corollary \ref{restriction} and by the fact that minimality conditions (1), (2), (3) are stable under base change, we know that these are first order representations for $\mathcal{C}_{j}$, so there are invertible matrices $S_{j}$, $T_{j}$ such that
\begin{equation*}
(K'_{j},L'_{j},M'_{j})=(T_{j}K_{j}S^{-1}_{j},T_{j}L_{j}S^{-1}_{j},T_{j}M_{j})
\end{equation*}
By Lemma \ref{properties} (2) we know that there are unique invertible matrices over $R$, say $S$ and $T$, such that $S \ mod(I_{j})=S_{j}$ and $T \ mod(I_{j})=T_{j}$. Obviously $K'=TKS^{-1}$, $L'=TLS^{-1}$ and $M'=TM$.
\end{proof}
We conclude with our main result:
\begin{theorem}\label{pupuri}
Let $ R_{1},\hdots, R_{t}$ be a commutative rings. If $R_{j}$ is a ring with representations for all $j=1,\ldots,t$ then $R=R_{1}\times\hdots\times R_{t}$ is a ring with representations.
\end{theorem}
\begin{proof}
It is clear from Propositions \ref{existencia2} and \ref{unicidad2}.
\end{proof}
Recall that the only rings with representations known so far are finite fields. Thus Theorem \ref{pupuri} implies that finite products of finite fields (reduced noetherian rings with Krull dimension $0$) are also rings with representations. 

\section{Input/Space/Output Representations of Families of Convolutional Codes over Finite Products of Finite Fields}

In this section, we specialize to the case $R_{i}=\mathbb{F}_i$ is a field for each $i=1,\hdots,t$. Let us show that we can generalize the construction of I/S/O representations given in (\ref{iso}) to families of convolutional codes over finite product of finite fields, $R=\mathbb{F}_{1}\times\hdots\times\mathbb{F}_{t}$.

\begin{theorem} \label{vamosquenosvamos}
 Let $\mathfrak{C}$ be a convolutional code over $R$. Let $(K,L,M)$ be a first order representation of $\mathfrak{C}$. Then
\begin{enumerate}
\item[i)] We can make elementary transformations over $(K,L,M)$ and obtain $(\mathcal{K}, \mathcal{L},\mathcal{M})$ such that 
\begin{equation}
\label{isoring}
\mathcal{K}=\left(\begin{array}{cc}
-Id_{\delta}\\
O \\
\end{array}
\right), \mathcal{L}=\left(\begin{array}{cc}
A \\
C \\
\end{array}
\right) \textit{ and }  \mathcal{M}=\left(\begin{array}{cc}
O & B  \\
-Id_{(n-k)} & D\\
\end{array}
\right)
\end{equation}
where $A \in \mathcal{M}_{\delta \times \delta}(R)$, $B \in \mathcal{M}_{\delta \times k}(R)$, $C \in \mathcal{M}_{(n-k) \times \delta}(R)$ and $D \in \mathcal{M}_{(n-k) \times k}(R)$.
\item[ii)] Moreover, the triple of matrices obtained in $i)$ verifies that
\begin{center}
$Ker(zK+L\mid M) \simeq Ker (z\mathcal{K}+\mathcal{L} \mid \mathcal{M})$
\end{center}
\end{enumerate} 
\end{theorem}
\begin{proof}
 $i)$ Since $(K \mid L \mid M)$ is surjective, we know that $(K_j \mid L_j \mid M_j)$ has full rank (see Corollary \ref{fullrank}) and therefore there exists an invertible matrix $W$ of size $(\delta+n-k)$ (see Lemma \ref{lema}) such that if we multiply $W^{-1}$ and $(K \mid L \mid M)$ and we reorder the codewords of the local codes $\mathcal{C}_j$, if it is necessary, we get a triple of matrices $(\mathcal{K} \mid \mathcal{L} \mid \mathcal{M})$ in the following way

\begin{center}
$\mathcal{K}=\left(\begin{array}{cc}
-Id_{\delta}\\
O \\
\end{array}
\right), \mathcal{L}=\left(\begin{array}{cc}
A_{\delta \times \delta} \\
C_{(n-k) \times \delta} \\
\end{array}
\right)$ and $ \mathcal{M}=\left(\begin{array}{cc}
O & B _{\delta \times k} \\
-Id_{(n-k)} & D_{(n-k) \times k} \\
\end{array}
\right)$
\end{center}

  $ii)$ Follows from \cite[Proposition 2.5]{nuestroarxiv}.
\end{proof}

% \begin{remark}
%The procedure developed in Section \ref{sectionISOFIELDS} for fields holds to $R$ being the way to get  (\ref{dinamic2}) from (\ref{rosa}).
% \end{remark}
\begin{proposition}
  \label{noenoe}
 The matrices $(A,B,C,D)$ over $R$ obtained in (\ref{isoring}) can be constructed from  $(A_{j},B_{j},C_{j},D_{j})$ over $R_{j}$, the I/S/O representations of the convolutional codes  $\mathcal{C}_{j}$ over $\mathbb{F}_{j}$.
  \end{proposition}
\begin{proof}
Let $W$ be the invertible matrix of size $(\delta+n-k)$ defined in  Proof of $i)$ of Theorem \ref{vamosquenosvamos}. Let $W_{j}\equiv W$ $(mod$ $I_{j})$ be the square minor of size $(\delta+n-k)$ such that $det(W_{j}) \in \mathbb{F}_{j}^{*}$. 

We know that 
\begin{equation}
\label{minors}
W^{-1}(K \mid L \mid M)= (\mathcal{K} \mid \mathcal{L} \mid \mathcal{M}) 
\end{equation}
where $(K \mid L \mid M)$ and $(\mathcal{K} \mid \mathcal{L} \mid \mathcal{M})$ are minimal first order representations of $\mathfrak{C}$.

If we apply $mod$ $I_{j}$ to both sides of above equation we get
\begin{equation}
\label{oeste}
[W^{-1}\cdot (K \mid L \mid M)](\textit{mod } I_{j})=(\mathcal{K}_{j} \mid \mathcal{L}_{j} \mid \mathcal{M}_{j})=\left[\begin{array}{cc}
-Id_{j}\\
O \\
\end{array}
  \begin{array}{cc}
A_{j} \\
C_{j} \\
\end{array} \begin{array}{cc}
O & B_{j} \\
-Id_{j} & D_{j} \\
\end{array}\right]
\end{equation}
On the other hand 
\begin{equation}
\label{este}
(\mathcal{K} \mid \mathcal{L} \mid \mathcal{M})(\textit{mod } I_{j})= \left[\begin{array}{cc}
-Id\\
O \\
\end{array}
  \begin{array}{cc}
A \\
C\\
\end{array} \begin{array}{cc}
O & B\\
-Id & D \\
\end{array}\right](\textit{mod } I_{j})
\end{equation}
and since the equation (\ref{minors}) is verified then we conclude the proof.
\end{proof}

Recall that a a linear system $\Sigma=(A,B,C,D)$ over a commutative ring $R$ is reachable (controllable from the origin) if the controllability matrix $\Phi_{\delta}(A,B):= (B, AB, . . . , A^{\delta-1}B)$ is surjective. Let as show that the I/S/O representation constructed above is always reachable.

\begin{proposition}
\label{hastalos}
Let $R= \mathbb{F}_{1} \times \ldots \times \mathbb{F}_{t}$ be a commutative ring where $\mathbb{F}_{j}$ is a finite field for all $j=1,\ldots,t$ and let $\Sigma^{\mathfrak{C}}$ be the I/S/O representation of a family of convolutional codes $\mathfrak{C}$ over $R$. Then $\Sigma^{\mathfrak{C}}$ is a reachable linear system over $R$.
\end{proposition}
\begin{proof}
Let $(A,B,C,D)$ be an I/S/O representation of a $\mathfrak{C}$ over $R$. By Proposition \ref{noenoe} we can consider these matrices such that
\begin{equation*}
A_{j} \equiv A (\textit{mod } I_{j}), B_{j} \equiv B (\textit{mod } I_{j}), C_{j} \equiv C (\textit{mod } I_{j}) \textit{ and } D_{j} \equiv D (\textit{mod } I_{j}).
\end{equation*}
Now, by \cite{York}, each I/S/O representation of each convolutional code over $\mathbb{F}_{j}$, $\Sigma_{j}^{\mathcal{C}_{j}}$, verifies that $rank$ $\Phi_{\delta}(A_{j},B_{j})=\delta$ and so, they are reachable linear systems. Since $R$ is a pointwise ring, $\Phi_{\delta}(A,B)$ is surjective (see Lemma \ref{lema}) and we conclude that the I/S/O of $\mathfrak{C}$, $\Sigma^{\mathfrak{C}}$, is a reachable linear system over $R$.
\end{proof}

\begin{example}
\label{ejuno}
In this example, we will raise two encoders with coefficients in $\mathbb{Z}/2\mathbb{Z}$ and $\mathbb{Z}/3\mathbb{Z}$, each one generating a dynamical linear system. We compute their first order and I/S/O representations. Consider the encoder on $\mathbb{Z}/2\mathbb{Z}$
\begin{center}
$G_{1}(z)=\left(\begin{array}{cc}
z-1 & 1 \\
z^{2}+1 & 0 \\
z^{2}+1 & z+1
\end{array}
\right)$
\end{center}
Then there exist matrices $K_{1}, L_{1}$ and $M _{1}$ that characterize the encoder $G_{1}(z)$. 
\begin{equation*}
K_{1}=\left(\begin{array}{ccc}
1 & 0 & 0 \\
1 & 0 & 0 \\
0 & 1 & 0 \\
0 & 1 & 1 \\
\end{array}
\right), L_{1}=\left(\begin{array}{ccc}
0 & 1 & 0 \\
1 & 0 & 1 \\
1 & 0 & 0 \\
1 & 0 & 1 \\
\end{array}
\right)
\textit{ and } 
M_{1}=\left(\begin{array}{ccc}
0 & 0 & 0 \\
1 & 0 & 0 \\
0 & 1 & 0 \\
0 & 0 & 1 \\
\end{array}
\right)
\end{equation*}
Therefore we compute the matrices $A_{1}, B_{1}, C_{1}$ and $D_{1} $, by which one gets the associated linear system.
\begin{equation*}
A_{1}=\left(\begin{array}{ccc}
0 & 1 & 0 \\
1 & 0 & 0 \\
0 & 0 & 1 
\end{array}
\right), 
B_{1}=\left(\begin{array}{cc}
0 & 0 \\
1 & 0 \\
1 & 1
\end{array}
\right),
C_{1}=\left(\begin{array}{ccc}
1 & 1 & 1 
\end{array}
\right) \textit{ and } D_{1}=\left(\begin{array}{cc}
0 & 0  
\end{array}
\right)
\end{equation*}
Now note that
$$rk\left(\begin{array}{ccc} B_{1} & A_{1}B_{1} & A_{1}^{2} B_{1} \\ \end{array}\right)= rk \left(\begin{array}{cccccc} 0 & 0 &1&0&0&0 \\ 1&0&0&0&1&0 \\ 1&1&1&1&1&1 \\  \end{array} \right)= 3$$
so, $(A_1,B_1,C_1,D_1)$ is reachable. \\
Consider the encoder on $\mathbb{Z}/3\mathbb{Z}$
\begin{center}
$G_{2}(z)=\left(\begin{array}{cc}
z & -1 \\
1 & z-1 \\
-z^{2}+z-1 & 0
\end{array}
\right)$
\end{center}
Matrices $K_{2}, L_{2}$ and $M_{2}$ are obtained following:
\begin{equation*}
K_{2}=\left(\begin{array}{ccc}
-1 & 0 & 0 \\
-1 & 0 & 0 \\
0 & 0 & -1 \\
-1 & 1 & 0 \\
\end{array}
\right),
 L_{2}=\left(\begin{array}{ccc}
0 & 1 & 0 \\
0 & 0 & 1 \\
-1 & 0 & 1 \\
1 & 0 & 0 \\
\end{array}
\right) \textit{ and } 
M_{2}=\left(\begin{array}{ccc}
0 & 0 & 0 \\
1 & 0 & 0 \\
0 & 1 & 0 \\
0 & 0 & 1 \\
\end{array}
\right)
\end{equation*}
Then
\begin{equation*}
A_{2}=\left(\begin{array}{ccc}
0 & 1 & 0 \\
-1 & 1 & 0 \\
-1 & 0 & 1 
\end{array}
\right), 
B_{2}=\left(\begin{array}{cc}
0 & 0 \\
0 & -1 \\
1 & 0
\end{array}
\right),
C_{2}=\left(\begin{array}{ccc}
0 & 1 & -1 
\end{array}
\right), \textit{ and }
D_{2}=\left(\begin{array}{cc}
0 & 0  
\end{array}
\right)
\end{equation*}
Note again that
$$rk\left(\begin{array}{ccc} B_{2} & A_{2}B_{2} & A_{2}^{2} B_{2} \\ \end{array}\right)= rk \left(\begin{array}{cccccc} 0 & 0 & 0& 2&0&2   \\ 0&2&0&2&0&0 \\ 1&0&1&0&1&1 \\  \end{array} \right)= 3$$ 
so $(A_2,B_2,C_2,D_2)$ is also reachable. \\
Now, we obtain the corresponding matrices $(A,B,C,D)$ over $\mathbb{Z}/6\mathbb{Z}$ \emph{glueing} the matrices $(A_{1},B_{1},C_{1},D_{1})$ and $(A_{2},B_{2},C_{2},D_{2})$
\begin{equation*}
A=\left(\begin{array}{ccc}
0 & 1 & 0 \\
5 & 4 & 0 \\
2 & 0 & 1 
\end{array}
\right), 
B=\left(\begin{array}{cc}
0 & 0 \\
3 & 2 \\
1 & 3
\end{array}
\right),
C=\left(\begin{array}{ccc}
3 & 1 & 5 
\end{array}
\right), and D=\left(\begin{array}{cc}
0 & 0  
\end{array}
\right)
\end{equation*}
The associated I/S/O representation of $\mathfrak{C}$ over $\mathbb{Z}/6\mathbb{Z}$ is reachable too, since 
$$\mathcal{U}_{3}(\Phi_{3}(A,B))=<\left|\begin{array}{ccc} 3 & 2 & 2 \\ 0 & 2 & 0 \\ 1 & 3 & 1 \\ \end{array}\right|,\left|\begin{array}{ccc} 3 & 0 & 2 \\ 0 & 3 & 0 \\ 1 & 1 & 1 \\ \end{array}\right|>=<2,3>=\mathbb{Z}/6\mathbb{Z}$$
We also perform matrices $K, L$ and $M$ in $ \mathbb{Z}/6\mathbb{Z}$ 
\begin{equation}
\label{KLM6}
K=\left(\begin{array}{ccc}
-1 & 0 & 0 \\
0 & -1 & 0 \\
0 & 0 & -1 \\
0 & 0 & 0 \\
\end{array}
\right), L=\left(\begin{array}{ccc}
0 & 1 & 0 \\
5 & 4 & 0 \\
2 & 0 & 1 \\
3 & 1 & 5 \\
\end{array}
\right)
\textit{ and } 
M=\left(\begin{array}{ccc}
0 & 0 & 0 \\
0 & 3 & 2 \\
0 & 1 & 3 \\
-1 & 0 & 0 \\
\end{array}
\right)
\end{equation}
Now we compute $Ker(zK+L|M)$ in order to obtain a encoder of a family of convolutional codes over 
$\mathbb{Z}/6\mathbb{Z}[z]$,
\begin{equation}
\label{encoder6}
G(z)=\left(\begin{array}{cc}
z+3 & 5 \\
3z^{2}+1 & -2z+2 \\
-z^{2}+4z-1 & 3z-3
\end{array}
\right)
\end{equation}
 Note that above encoder $G(z)$ restricts to $\mathbb{Z}/2\mathbb{Z}$ obtaining $G_{1}(z)$ and $\mathbb{Z}/3\mathbb{Z}$ getting $G_{2}(z)$.
\end{example}

\section{Construction of Observable Families of Convolutional Codes over Finite Product of Finite Fields.}
Let $\mathbb{F}_{1},\hdots,\mathbb{F}_{t}$ be finite fields and consider the ring $R=\prod\mathbb{F}_{i}$ which is a ring with representations, as we have already shown. It is well known that if we consider a reachable and observable I/S/O representations over a finite field then we get an observable convolutional code by minimal first order representation by the formula (\ref{iso}) (see \cite{RSY} and \cite{York}  for details). In this section we show that we can generalize this result to the case of I/S/O representations of a family of convolutional codes over $R$.

Recall that a linear system $\Sigma=(A,B,C,D)$ over a commutative ring $R$ is observable if the observability matrix $\Omega_{\delta}(A,C):=(C^{t}, (CA)^{t}, (CA^{2})^{t}, \ldots, (CA^{\delta-1})^{t})$ is injective.
%recall some basic results:
%\begin{remark}[cf. \cite{BBvV}]
%Let $\Sigma=(A,B,C,D)$ be a linear system over a commutative ring $R$ with $X=R^{\delta}$ the state space. Recall that $A\in Mat_{\delta\times\delta}(R)$. 
%\begin{enumerate}
%\item The linear system $\Sigma$ is reachable (controllable from the origin) if the controllability matrix $\Phi_{\delta}(A,B):= (B, AB, . . . , A^{\delta-1}B)$ is surjective.
%\item The linear system $\Sigma$ is observable if the observability matrix $\Omega_{\delta}(A,C):=(C, CA, CA^{2}, \ldots, CA^{\delta-1})^{t}$ is injective.
%\end{enumerate}
%\end{remark}
Now we give the result that allow us to construct observable families of convolutional codes from observable I/S/O representations over $R$.
\begin{proposition}
Let $R = \mathbb{F}_{1} \times \ldots \times \mathbb{F}_{t}$ be a commutative ring where $\mathbb{F}_{j}$ is a finite field for all $j=1,\ldots,t$ and let $\Sigma^{\mathfrak{C}}$ be the I/S/O representation of a family of convolutional codes $\mathfrak{C}$ over $R$. If $\Sigma^{\mathfrak{C}}$ is a reachable and observable linear system over $R$ then $\mathfrak{C}$ is an observable family of convolutional codes over $R$ 
%being $\mathfrak{C}$ the family of convolutional codes over $R$ obtained considering $(A,B,C,D)$ as the I/S/O representations of $\mathfrak{C}$.
\end{proposition}
\begin{proof}
By hypothesis $(A,B,C,D)$ is a reachable and observable linear system over $R$, so $\Omega_{\delta}(A,C)$ is injective and $\Phi_{\delta}(A,B)$ is surjective. Consider $(A_{j},B_{j},C_{j},D_{j})$, the linear systems obtained over each $\mathbb{F}_{j}$ for $j=1,\ldots,t$. Clearly $\Phi_{\delta}(A_{j},B_{j})$ is also surjective for all $j$. Since $R$ is an absolutely flat ring, the cokernel of $\Omega_{\delta}(A,C)$ is flat over $R$. Therefore $\Omega_{\delta}(A_{j},C_{j})$ is also injective for all $j$. Thus, for all $j$ the above systems are reachable and observable too. Then, if we perform the convolutional codes $\mathcal{C}_{j}$ for each $j$ from $(A_{j},B_{j},C_{j},D_{j})$, by \cite{York} $\mathcal{C}_{j}$ are observable convolutional codes for all $j$. By Lemma \ref{observableiff} then $\mathfrak{C}$ is an observable family of convolutional codes over $R$.
\end{proof}

\section{Conclusions}

We have proved the existence and uniqueness of minimal first order representations of families of convolutional codes over certain commutative rings. This property defines the class of rings that we have called rings with representations. For instance, we have shown that finite products of finite fields belong to this class. Here a natural question comes up: are infinite products of finite fields $\prod \mathbb{F}_{i}$ rings with representations? Show that these are special cases of von Neumann regular rings. So the next natural question is: are von Neumann regular rings rings with representations? 

In the particular case of finite product of finite fields, we also generalize the existence of I/S/O representations and we construct observable families of convolutional codes from linear systems.

Our further research is focused on answering the above questions and to get I/S/O representations for this type of rings.

%\section{Conclusions}

\end{document}